\documentclass[11pt,english]{smfart}

\usepackage[french,main=english]{babel}
\usepackage[utf8, latin1]{inputenc} 
\usepackage[T1]{fontenc}
\usepackage{stmaryrd,ae,euscript,enumerate,mathptmx,newtxtext,newtxmath,dsfont,a4wide}
\usepackage[cm]{aeguill}   

\newcommand{\inv}{^{\raisebox{.2ex}{$\scriptscriptstyle-1$}}}   

\usepackage[dvipsnames]{xcolor}   
\usepackage{xr-hyper}
\definecolor{brightmaroon}{rgb}{0.76, 0.13, 0.28}
\usepackage[linktocpage=true,colorlinks=true,hyperindex,citecolor=Royal Blue,linkcolor=Royal Blue]{hyperref}

\newtheorem{theorem}{Theorem}[section]
\newtheorem{proposition}[theorem]{Proposition}
\newtheorem{corollary}[theorem]{Corollary}
\newtheorem{lemma}[theorem]{Lemma}
\newtheorem{definition}[theorem]{Definition}
\newtheorem{remark}[theorem]{Remark}
\numberwithin{equation}{section}   

\newtheorem{conrem}{Concluding remarks}
 
\usepackage[all]{xy}
\usepackage{dsfont}  

\begin{document}

\title{Spectral topologies on skew braces}

\author{Themba Dube}

\address{(1) Department of Mathematical Sciences, University of South Africa,
P.O. Box 392, 0003 Pretoria, South Africa.
(2) National Institute for Theoretical and Computational Sciences (NITheCS), South Africa.}
\email{dubeta@unisa.ac.za}

\author{Amartya Goswami}
\address{(1) Department of Mathematics and Applied Mathematics, University of Johannesburg, P.O. Box 524, Auckland Park 2006, South Africa.
(2) National Institute for Theoretical and Computational Sciences (NITheCS), South Africa.}
\email{agoswami@uj.ac.za}




\begin{abstract}
Using a new definition of a prime ideal of a skew brace $A,$  on set $\mathrm{Spec}\,A$ of prime ideals of $A$ we endow a spectral topology (in the sense of \textsc{Grothendieck}). We characterize irreducible closed subsets of $\mathrm{Spec}\,A$ and prove every irreducible closed subset of the space has a unique generic point. We give a sufficient condition for the space to be Noetherian. We study continuous maps between such spaces, and finally, we prove that $\mathrm{Spec}(\mathrm{Idl}\,A)$ is a spectral space, where $\mathrm{Idl}\,A$ is the set of all ideals of $A$.
\end{abstract}

\maketitle
\section{Introduction}

To study non-degenerate involutive set-theoretic solutions of the Yang--Baxter equation, \textsc{Rump} introduced a new algebraic structure in \cite{R07} called a (left) brace, which was generalized into noncommutative setting by \textsc{Guarnieri \& Vendramin} in \cite{GV17} and called a (left) skew brace. Since the introduction of skew braces, several algebraic properties of them have been studied. \textsc{Jespers, Kubat,Van Antwerpen, \& Vendramin} studied factorization of skew left braces through strong left ideals in \cite{JKAV19} and proved analogs of It\^{o}'s theorem in the context of skew left braces, whereas in \cite{JKAV21}, they obtained a Wedderburn type decomposition for Artinian skew
left braces and also proved analogues of a theorem of
Wiegold, a theorem of Schur and its converse. In \cite{CSV19}, \textsc{Ced\'{o}, Smoktunowicz \& Vendramin} studied series of left ideals of skew left braces that are analogs of upper central series of groups and using that they defined left and right nilpotent skew left braces. The concepts like ideals, series of ideals, prime, and semiprime
ideals, Baer and Wedderburn radicals and solvability for skew braces have been explored in \cite{KSV21} by \textsc{Konovalov, Smoktunowicz \& Vendramin}.

Skew braces are fairly new algebraic structures and much more study of their algebraic aspects yet to come. Compared to algebraic properties studied so far, little has been explored from the topological side. The aim of this paper is to initiate that aspect and we do so by studying some of the topological properties of the spectrum of prime ideals of a skew brace. It is worth it to mention that in \cite{KSV21}, it has been pointed out that ``in the conference `Groups, rings, and the Yang--Baxter
equation,' Spa, 2017, Louis Rowen suggested that it could be interesting to study prime ideals of skew braces.'' To meet our requirements and to make analogy with the prime spectrum of a (commutative) ring, we propose a new definition of prime ideals of a skew brace. We discuss separation axioms and irreducubility. Using a skew brace homomorphism, we obtain continuous maps between spaces of spectra and study properties of these maps. Finally, we prove a result on spectrality. We end the paper with some open questions.

\section{Preliminaries}   

A \emph{skew left brace} is a triple $\left(A, +, \circ\right)$ such that $\left(A, +\right)$ and $\left(A, \circ\right)$ are groups, and the operations satisfy the identity:  
$$
a\circ\left(b + c\right) = a\circ b - a + a\circ c,
$$
for all $a, b, c$ in $A$, where $-a$ denotes the inverse of $a$ with respect to the $+$ operation. 
By a skew brace, we will always mean a 
skew left brace. We say that $\left(A, +\right)$ is the additive  group of $A$, whereas $\left(A, \circ\right)$ is the multiplicative group of $A$. The identity of the additive group of $A$ coincides with that of the multiplicative group  of $A,$ and we denote the common identity element by $e$. The  multiplicative inverse of an element $a$ in $A$ is denoted by $a\inv$. We write $A$ as a skew brace to mean the triple $\left(A,+,\circ\right)$. 

Given a skew brace $A$, a new operation $*$ is defined as follows: 
$$
a* b = - a + a\circ b  - b,   
$$
for all $a,b$ in $A$.
In particular, $a* b = \lambda_a\left(b\right) - b$, for all $a,b$ in $A$, where
$\lambda_a$ is an automorphism of $(A,+)$ and is defined by: $$\lambda_a(b)=-a+a\circ b.$$  

An \emph{ideal} $I$ of a skew brace $A$ is a normal subgroup of both the group structures $(A, +)$ and $(A, \circ)$ such that $\lambda_a(I)\subseteq I$ for all $a\in A$. We define the \emph{trivial} or \emph{zero} ideal of $A$ as the singleton set $\{e\}$ of $A$ and denote it by $0$, whereas by a \emph{proper} ideal $I$, we mean $I$ is an ideal of $A$ and $I\neq A.$ If $S$ is a subset of $A$, then $\langle S \rangle$ denotes the ideal generated by $S$. A \emph{maximal ideal} $M$ is a proper ideal of $A$ and not properly containing in another proper ideal of $A$. 
We denote the set of all ideals of $A$ by $\mathrm{Idl}\,A$, whereas the set of all maximal of a skew brace $A$ is denoted by $\mathrm{Max}\,A$. 
Involving the $*$ operation, a characterisation (see \textsc{Ced\'{o}, Smoktunowicz, \& Vendramin} \cite[Lemma 1.8, Lemma 1.9]{CSV19}) of normal subgroups of any skew brace to be an ideal is given in the following

\begin{proposition} \label{ali} 
 Suppose $A$ is a skew brace and $I$ is a normal subgroup of $\left(A,+\right)$. Then $I$ is an ideal of $A$ if and only if $A* I\subseteq I$ and $I* A\subseteq I$.
\end{proposition} 
 
\begin{lemma}\label{inid}
 If $\{I_{\lambda}\}_{\lambda \in \Lambda}$ is a family of ideals of $A$, then $\bigcap_{\lambda \in \Lambda} I_{\lambda}$ is also an ideal of $A$. If $I$ and $J$ are ideals of $A$, so is $I+J,$ where $I+J$ is defined as the additive subgroup of $A$ generated by $\{ i+j \mid i\in I,\, j\in J\}.$
\end{lemma}

\begin{proof}
Since $\{I_{\lambda}\}_{\lambda \in \Lambda}$ is a family of ideals, $\lambda_a(x)$ is in $\bigcap_{\lambda \in \Lambda} I_{\lambda}$ for all $a\in A$ and for all $x\in \bigcap_{\lambda \in \Lambda} I_{\lambda}.$ Since arbitrary intersection of normal subgroups is normal, it follows that $\bigcap_{\lambda \in \Lambda} I_{\lambda}$ is an ideal of $A$.
To prove that $I+J$ is an ideal, we first observe that $I+J$ is a normal subgroup of $(A,+)$. By Proposition \ref{ali}, it is sufficient to show that  $A*(I+J)\subseteq I+J$ and $(I+J)* A\subseteq I+J$. To prove the first inclusion, let $a\in A$, $i\in I$, $j\in J.$ Then we have
\begin{equation*}
\begin{split}
a*(i+j) &= -a+a\circ (i+j)-i-j\\
&=-a+a\circ i-a+a\circ j-i-j\\
&= \lambda_a(i)+\lambda_a(j) - (i+j) \in I+J, 
\end{split}
\end{equation*}
proving $A*(I+J)\subseteq I+J$. The proof of the containment $(I+J)* A\subseteq I+J$ follows from \textsc{Konovalov, Smoktunowicz, \& Vendramin}  \cite[Lemma 3.7]{KSV21}. 
\end{proof}
The binary sum of ideals of a skew brace can be extended to an ``infinite sum'', which we shall see in the next corollary. But before that, let us first define this infinite sum of ideals of a skew brace $A$. If $\{I_{\lambda}\}_{\lambda \in \Lambda}$ is a family of ideals of $A$, then by the \emph{sum} $\sum_{\lambda\in \Lambda} I_{\lambda}$ of ideals of $A$, we mean each term of the sum is over finitely many indices. 

\begin{corollary}\label{insum}
 If $\{I_{\lambda}\}_{\lambda \in \Lambda}$ is a family of ideals of $A$, then $\sum_{\lambda\in \Lambda} I_{\lambda}$ is also an ideal of $A$. 
\end{corollary}

\begin{remark}
\emph{	
$(1).$ Recall that  an associative ring $(R, +, \circ)$ is called a \emph{radical ring} if  $(R, *)$ is a group, where the operation $*$ is defined by $a*b=a+b-a\circ b.$  \textsc{Rump} \cite{R07} showed that $(R, +, *)$ is an example of a (left) \emph{brace}. It is now natural to ask what should be ``ring-type'' example of a skew brace, and here the problem starts. Since the addition operation of a ring is commutative, we can not expect an example even from a noncommutative ring. We rather need to consider a near-ring and apply the above method to have the expected example.} 

\emph{$(2).$ Note that if a ring $R$ is a radical ring, then every element of $R$ is quasi-regular (see \textsc{Jacobson} \cite{J56}) and such a ring can not have a multiplicative identity $1$. Indeed, for $x\in R$, 
$$x*1=x+1-x\circ 1=x+1-x=1\neq 0,$$
a contradiction. Therefore, it follows that we can not have an ``unital'' brace or a ``unital'' skew brace (unless the braces or the skew braces are trivial). Furthermore, we observe that the identity element $e$ of a skew brace $A$ can not play the role of an unital element. Otherwise, $A$ will not have any proper ideals (as by definition $e\in I$ for all ideals $I$ of $A$).}
\end{remark}

\section{Spectra of skew braces}

There are various ways to define prime ideals of  a skew brace. Here we describe three of them. 

\begin{definition}\label{ksv} 
Let $I$ and $J$ be two ideals of a skew brace $A.$ A proper ideal $P$ of $A$ is called \emph{prime} if   $$I* J\equiv 0\,(\mathrm{mod} \;\,P) \implies I\equiv 0\,(\mathrm{mod} \;\,P)\;\; \text{or}\;\; J\equiv 0\,(\mathrm{mod} \;\,P),$$ where $I* J$ is the additive subgroup of $A$ generated by $\{i* j \mid i\in I,\, j\in J\}.$ 
\end{definition}
 
The Definition \ref{ksv} above follows from \textsc{Konovalov, Smoktunowicz, \& Vendramin} \cite[Definition 5.8]{KSV21}, whereas Definition \ref{bfp} below is based on \textsc{Bourn, Facchini, \& Pompili} \cite[Proposition 3.18]{BFP22}, and has been proposed to the second author by \textsc{Alberto Facchini} in a private communication.
 
\begin{definition}
\label{bfp} 
\emph{Let $I$ and $J$ be any two ideals of $A.$ A proper ideal $P$ of $A$ is called \emph{prime} if  $$[I, J]\subseteq P \implies I\subseteq P\;\; \text{or}\;\; J\subseteq P,$$ where $[I, J]$ is the (Huq) commutator of the ideals $I$ and $J$  generated by the union of the following three
sets: 
\begin{enumerate}[\upshape (i)]
\item\label{ij+} $\{ i + j - i - j \mid i \in I, j \in J \};$  
\item\label{ijd} $\{ i\circ  j \circ  i\inv\circ   j\inv \mid i \in I, j \in J \};$ 
\item $\{ i\circ  j - j - i\mid \in I, j \in J\},$ 
\end{enumerate}
where the sets in (\ref{ij+}) and (\ref{ijd}) respectively generate  the commutators
$[I, J]_{(A,+)}$ and $[I, J]_{(A,\circ)}$ of the normal subgroups $I$ and $J$ of the groups $(A, +)$ and  $(A, \circ).$ }
\end{definition}

Finally, we propose the following definition of a prime ideal of a skew brace.

\begin{definition}
\label{pri} 
\emph{A proper ideal $P$ of $A$ is called \emph{prime} if, for any subsets I and J of A, $$I* J\subseteq P \implies I\subseteq P\;\; \text{or}\;\; J\subseteq P,$$ where $I* J=\{i* j \mid i\in I, \,j\in J\}$. }
\end{definition}  
 
We denote the set of all prime ideals of $A$ by $\mathrm{Spec}\,A$. For the rest of the work, by a prime ideal of a skew brace, we shall mean it in the sense of Definition \ref{pri} and this definition helps us in proving the preservation of prime ideals under inverse image of a skew brace homomorphism (see Proposition \ref{conmap}(\ref{sbcm})). Although in Definition \ref{pri} we have considered $I$ and $J$ just as subsets of $A$, nevertheless, we have the following implication, where with an abuse of notation, we shall also write $I*J$ even when $I$ and $J$ are ideals of a skew brace $A$.

\begin{proposition}\label{npd}
Suppose a skew brace $A$ is a skew brace and $P$ is a prime ideal of $A$.
If $I$ and $J$ are ideals of $A$, then $I*J\subseteq P$ implies $I\subseteq P$ or $J\subseteq P,$ where $*$ is the operation as in Definition \ref{pri} and $I*J$ is the ideal generated by the set $\{i*j\mid i\in I, j\in J\}.$
\end{proposition}

\begin{proof}
Suppose $\hat{I}$ and $\hat{J}$ respectively denote the underlying sets of the ideals $I$ and $J$.  Then
$$\hat{I}*\hat{J}\subseteq I*J\subseteq P.$$
Since $P$ is prime, either $\hat{I}\subseteq P$ or $\hat{J}\subseteq P$, and that implies either 
$I=\langle \hat{I} \rangle \subseteq P$ or $J=\langle \hat{J} \rangle \subseteq P$.
\end{proof}


\begin{lemma}\label{dint}  If $I$ and $J$ are any two ideals of a skew brace $A$, then $I* J\subseteq I\cap J.$
\end{lemma}
  
\begin{proof}
If $i* j\in I* J,$ then $i* j=\lambda_i(j)-j \in J$ as $J$ is an ideal. On the other hand,
\begin{equation*}
i* j 
= -i + i\circ  j -j
= -i + j\circ \left(j\inv\circ   i\circ  j\right) -j 
= -i + j + \lambda_{j}\left(j\inv\circ  i \circ j\right) -j\in I.\qedhere 
\end{equation*}  
\end{proof}
  
\begin{proposition}\label{mullat}
Suppose $A$ is a skew brace. Then $(\mathrm{Idl}\,A, \subseteq, *)$ is a multiplicative lattice.
\end{proposition}

\begin{proof}
Notice that under subset inclusion $\subseteq$ relations, the emptyset and $A$ are respectively the bottom and top elements of $\mathrm{Idl}\,A$. That the pair $(\mathrm{Idl}\,A, \subseteq)$ is a complete lattice now follows from Lemma \ref{inid} and Corollary \ref{insum}, whereas by Lemma \ref{dint} it follows that the complete lattice $(\mathrm{Idl}\,A, *)$ is also multiplicative.
\end{proof}

In a skew brace, maximal ideals need not be prime. For example, in a finite two sided brace $A$ with more than one element, each maximal ideal of $A$ cannot be prime. What we have instead is the following characterisation (see \textsc{Jespers, Kubat, Van Antwerpen, \& L. Vendramin} \cite[Proposition 3.6]{JKAV21}) of maximal ideals that are prime.

\begin{proposition}\label{minp}
 A maximal ideal $M$ of $A$ is  prime if and only if $A^{\scriptscriptstyle 2}\nsubseteq M,$ where $A^{\scriptscriptstyle 2}=A*A.$
\end{proposition}

Since the operation $*$ is neither associative nor commutative, an expression like $a^{\scriptscriptstyle n}=a*a^{\scriptscriptstyle n-1}$ does not make sense for us. Therefore, we cannot adapt the ``elementwise'' definition of radical of an ideal as we have for commutative rings. Following the definition of the radical of an ideal of a noncommutative ring (see \textsc{Lam} \cite[Theorem 10.7]{L01}), we propose  
 
\begin{definition}\label{radi}
\emph{The \emph{radical} of an ideal $I$ in a skew brace $A$, denoted by $\mathrm{Rad}\,I$, is defined as $$\mathrm{Rad}\,I=\bigcap_{I\subseteq P\,\in\, \mathrm{Spec}\,A} P.$$ 
The \emph{nil radical} $\mathrm{Nil}\,A$ of $A$ is the radical of the zero ideal (see Section 2) of $A$. So, $\mathrm{Nil}\,A$ is the intersection of all prime ideals of $A$.}
\end{definition}  

\section{Spectral topology}

Let $S$ be a subset of a skew brace $A$. If $H(S)=\{ P\in \mathrm{Spec}\,A \mid S\subseteq P\},$ then it is easy to see that $H(S)=H(\langle S \rangle).$ Therefore, it is sufficient to consider the set of all ideals of a skew brace brace $A$ to define closed sets of spectral topology on  $\mathrm{Spec}\,A$. The terminology ``spectral topology'' was first used by \textsc{Grothendieck} in \cite{G60}.
 
\begin{definition}\label{ztd}
\emph{A \emph{spectral}  topology on $\mathrm{Spec}\,A$ is imposed by considering the collection of sets $\{H(I)\}_{I\in \mathrm{Idl}\,A}$ as closed sets, where 
$$
H(I)=\{P\in \mathrm{Spec}\,A\mid I \subseteq P\}. 
$$}
\end{definition}

The following proposition shows that Definition \ref{ztd} indeed makes sense for skew braces.

\begin{proposition}\label{ztp}
The collection $\{H(I)\}_{I\in \mathrm{Idl}\,A}$ of sets  satisfies the following properties:
\begin{enumerate}[\upshape (i)]
\item\label{vza} $H(A)=\emptyset$ and $H(0)=\mathrm{Spec}\,A.$
\item\label{uis} $H(I)\cup H(J)= H(I\cap J)=H(I* J)$ for all $I, J\in \mathrm{Idl}\,A.$
\item\label{ais} $
\bigcap_{\lambda \in \Lambda}H(I_{\lambda})= H\left(\sum_{\lambda\in \Lambda} I_{\lambda}\right)
$ for all $I_{\lambda}\in \mathrm{Idl}\,A$ and $\lambda \in \Lambda.$
\end{enumerate}

\end{proposition}

\begin{proof} 
(i). Since by Definition \ref{pri}, $A\notin \mathrm{Spec}\,A,$ we have $H(A)=\emptyset.$ Since $\{e\}\subseteq P$ for all $P$ in $\mathrm{Spec}\,A$, we have  $H(0)=\mathrm{Spec}\,A.$ 
 
(ii). By Definition \ref{ztd}, it follows that $H(I)\cup H(J)\subseteq H(I\cap J)$ for any two ideals $I$ and $J$ of $A.$
On the other hand,  
$ H(I)\cup H(J)\supseteq H(I\cap J)$
follows by Lemma \ref{dint}. If $P\in H(I*J),$ then $I*J\subseteq P$, and by Proposition \ref{npd}, $I\subseteq P$ or $J\subseteq P,$ and that gives $P\in H(I)\cup H(J)$. The inclusion $H(I*J)\supseteq H(I)\cup H(J)$ follows from Lemma \ref{dint}.   

(iii). Note that if $\{I_{\lambda}\}_{\lambda \in \Lambda}$ is a family of ideals of $A$, then by Corollary \ref{insum}, $\sum_{\lambda\in \Lambda} I_{\lambda}$ is an ideal of $A$. By Definition \ref{ztd} we also have
$
\bigcap_{\lambda \in \Lambda}H(I_{\lambda})\supseteq H\left(\sum_{\lambda\in \Lambda} I_{\lambda}\right).
$
Conversely, if $P$ is in $\bigcap_{\lambda \in \Lambda}H(I_{\lambda}),$ then $I_{\lambda}\subseteq P$ for all $\lambda \in \Lambda$, and hence $\sum_{\lambda\in \Lambda} I_{\lambda}\subseteq P.$ This proves that $
\bigcap_{\lambda \in \Lambda}H(I_{\lambda})\subseteq H\left(\sum_{\lambda\in \Lambda} I_{\lambda}\right).
$
\end{proof}

\begin{lemma}\label{hihr}
If $I\in \mathrm{Idl}\,A$, then $H(I)=H(\mathrm{Rad}\,I).$
\end{lemma}

\begin{proof}
The proof follows by repeated applications of Definition \ref{radi}. 
\end{proof}

\begin{theorem}
For a subset $S$ of $\mathrm{Spec}\,A$, let $K(S)=\bigcap_{I\in S} I.$ The operator $K$  has the following properties. 
\begin{enumerate}[\upshape (i)]
\item  $K(\emptyset)=R$ and  $K\left( \bigcup_{\lambda \in \Lambda}T_{\lambda} \right)=\bigcap_{\lambda \in \Lambda} K\left(T_{\lambda}\right). $
\item If $T$ is a subset of $X$ and  $I$ is an ideal of $R$ then $KH(I)= \mathrm{Rad}\,I,$ and $HK(S)$ is the closure of $S$ in $\mathrm{Spec}\,A$.  
\end{enumerate}  
\end{theorem}

\begin{proof} 
(i). The first assertion follows from the empty intersection property. For the second, the fact $S_{\lambda} \subseteq \bigcup_{\lambda \in \Lambda}S_{\lambda}$ implies $K(S_{\lambda}) \supseteq K(\bigcup_{\lambda \in \Lambda}S_{\lambda}),$ and hence $\bigcap_{\lambda \in \Lambda}K(S_{\lambda}) \supseteq K(\bigcup_{\lambda \in \Lambda}S_{\lambda}).$ For the other half of the inclusion, let $S_{\lambda}=\{I_{\alpha, \lambda}\}_{\alpha \in L}$ and let $x \in \bigcap_{\lambda \in \Lambda}K(S_{\lambda})$. Then $x \in \bigcap_{\lambda \in \Lambda}(\bigcap_{\alpha \in L}I_{\alpha, \lambda});$ whence $x \in K(\bigcup_{\lambda \in \Lambda}S_{\lambda}).$

(ii). For the first assertion we observe that   $$a\in KH(I)\Leftrightarrow a\in P\;\text{for every}\; P \;\text{with}\;P\in H(I)\Leftrightarrow a\in P\;\text{for every}\; P\supseteq I.$$  
Therefore, $KH(I)=\bigcap_{P\supseteq I}=\mathrm{Rad}\,I.$ 
For the second claim, if a closed set $H(X)$ (for some subset $X$ of $A$) contains $S,$ then $X\subseteq I$ for all $I\in S;$ which subsequently implies $X\subseteq K(S)$ and hence $H(X) \supset HK(S)$. Since $T\subseteq HK(S),$ and $HK(S)$ is the smallest closed set of $X$ containing $S$, we have the desired claim.  
\end{proof}

Following \textsc{Kock} \cite{K07}, we can represent the relation between $H$ and $K$ categorically as follows:  
We observe that the poset map $K$ is a right adjoint of the map $H$. The unit of the adjunction $H \dashv K$ is $\eta \colon \left\langle S\right\rangle \mapsto KH(S)=  \left\langle S\right\rangle.$ Hence the full subcategory $\mathrm{Fix}\,\eta=\{S\in \mathfrak{P}(A)\mid \eta_S \;\text{is an isomorphism}   \}$ is the set of ideals of $R$. The counit of the adjunction $H \dashv K$ is $\epsilon \colon T \mapsto HK(T)=  \overline{T}.$ Therefore $\mathrm{Fix}\,\epsilon=\{T\in \mathfrak{P}(\mathrm{Spec}\,A)^{\text{op}}\mid \epsilon_T \;\text{is an isomorphism}   \}$ is the set of closed subsets.
Therefore the adjunction $H\dashv K$ restricts to an adjoint equivalence between categories $\mathrm{Idl}\,A$ and $\mathrm{Closed}(\mathrm{Spec} \,A).$ By considering open sets of the topology on $\mathrm{Spec}\,A$, the above isomorphism of categories become $\mathrm{Idl}\,A \approx \mathrm{Open}(\mathrm{Spec}\,A).$ The following diagram summarises the above inter-relations. 
\begin{equation*}\label{ids}
\vcenter{\xymatrix@C=2cm{
&\mathfrak{P}(A) \ar@<1.2ex>[r]^(.4){H} & \mathfrak{P}(\mathrm{Spec}\,A)^{\text{op}}\ar@<1ex>[l]^(.6){K}_(.6){\bot} \\ 
 \mathrm{Fix}\,\eta=\hspace*{-2.8cm}& \mathrm{Idl}\,A\ar@{^{(}->}[u]\ar@<1.1ex>[r] &  \mathrm{Closed}(\mathrm{Spec} \,A)^{\text{op}}\ar@<1ex>[l]_(.6){\approx}\ar@{^{(}->}[u] &\hspace*{-2.8cm}=\mathrm{Fix}\,\epsilon.  
}}  
\end{equation*}

\begin{lemma}\label{T0lem}
 Suppose $A$ is a skew brace and $P, Q\in \mathrm{Spec}\,A$. Then
\begin{enumerate}[\upshape (i)]
\item\label{clpv} 
$\mathrm{cl}{\{P\}}=H(P)$.
\item\label{pclq} \label{subprime} 
$P\in \mathrm{cl}{\{Q\}}$ if and only if $P\in H(Q)$.
\end{enumerate} 
\end{lemma}

\begin{proof}
(i). Since $H(P)$ is a closed set containing $P$ and $\mathrm{cl}{\{P\}}$ is the smallest closed set containing $P$, we immediately have $\mathrm{cl}{\{P\}}\subseteq H(P).$
For the converse, suppose  $\mathrm{cl}{\{P\}}=H(I)$ for some ideal $I$ of $A$. Clearly, $P\in \mathrm{cl}{\{P\}}$, hence $I\subseteq P$. Thus any prime ideal that contains $P$ will also contain $I.$ Hence, $H(P)\subseteq H(I)=\mathrm{cl}{\{P\}}$.
 
(ii). Observe that by (\ref{clpv}), $P\in \mathrm{cl}{\{Q\}}$ if and only if $P \in H(Q)$, and this occurs if and only if $Q \subseteq P$.
\end{proof}

\begin{proposition}\label{t0a}
Every $\mathrm{Spec}\,A$ is a $T_{0}$-space.
\end{proposition}

\begin{proof}
Let $P$ and $Q$ be two distinct points of $\mathrm{Spec}\,A$ such that $P\nsubseteq Q.$ By Lemma \ref{T0lem} we have
$Q\notin \mathrm{cl}{\{P\}}=H(P)\ni P,$ 
which implies there exists a closed set $H(P)$ containing $P$ that does not contain $Q$.
\end{proof}

With a further restriction on a skew brace, we obtain the following

\begin{proposition}
If $A$ is a skew brace with $A^{\scriptscriptstyle 2}\nsubseteq M$ for all $M\in \mathrm{Max}\,A,$ then $\mathrm{Spec}\,A$ is a $T_{1}$-space if and only if\, $\mathrm{Spec}\,A=\mathrm{Max}\,A$.
\end{proposition}

\begin{proof}
By Theorem \ref{minp} it follows that $\mathrm{Spec}\,A\supseteq \mathrm{Max}\,A$. Let $\mathrm{Spec}\,A\neq \mathrm{Max}\,A$. Then there exists $P\in \mathrm{Spec}\,A$ such that $P\notin \mathrm{Max}\,A$, and that implies there exists $M\in \mathrm{Max}\,A$ such that $P \subsetneq M$. Then $M\in H(P) = \mathrm{cl}{\{P\}}$, where the equality follows from Lemma \ref{T0lem}(\ref{clpv}). Hence $\{P\}$ is not closed and hence $\mathrm{Spec}\,A$ is not a $T_{1}$-space.
Conversely, suppose that $\mathrm{Spec}\,A = \mathrm{Max}\,A$. Then $H(P) = \{P\}$ for all $P\in \mathrm{Spec}\,A$ and hence every singleton set is closed, and that implies $\mathrm{Spec}\,A$ is a $T_{1}$-space.
\end{proof} 

Let us recall the notion of irreducible topological spaces and some of their properties. Readers may consult \textsc{G\"{o}rtz \& Wedhorn} \cite{GW20} for further details on irreducibility. 

\begin{definition}
\emph{If $X$ is a topological space, a closed subset $S$ is \emph{irreducible} if $S$ is not the union of two properly smaller closed subsets $S_{1}, S_{ 2}\subsetneq S.$  A maximal irreducible subset of a topological space $X$ is called an
\emph{irreducible component} of $X.$ A point $x$ in a closed subset $S$ is called a \emph{generic point} of $S$ if $S = \mathrm{cl}(\{x\}),$ the closure of $\{x\}$.}
\end{definition}

Recall the following basic properties on irreducibility.

\begin{proposition}\label{irrpp}
Let $X$ be a topological space.
\begin{enumerate}[\upshape (i)]
\item\label{ircl}  A subspace $Y$ of $X$ is irreducible if and only
if $\mathrm{cl}(Y)$ is irreducible. 
\item\label{mircv} The irreducible components of $X$ are closed and cover $X$.
\end{enumerate}
\end{proposition}
 
\begin{lemma}\label{lemprime}
$\{H(P)\}_{P\in \mathrm{Spec}\,A}$ are precisely the irreducible closed subsets of $\mathrm{Spec}\,A$.  
\end{lemma}
  
\begin{proof}
Let $P\in \mathrm{Spec}\,A$. Since $\{P\}$ is irreducible, by Proposition \ref{irrpp}(\ref{ircl}) and Lemma \ref{T0lem}(\ref{clpv}), so is  $H(P).$  Let $H(I)$ be an irreducible closed subset of of $\mathrm{Spec}\,A$ and if possible, $I\notin \mathrm{Spec}\,A.$ This implies there exist $I_1$ and $I_2$ such that  $I_{\scriptscriptstyle 1}\nsubseteq I$ and $I_{\scriptscriptstyle 2}\nsubseteq I$, but $I_{\scriptscriptstyle 1}*I_{\scriptscriptstyle 2}\subseteq I$. Then we have:  
$$H(\langle I, I_{\scriptscriptstyle 1}\rangle)\cup H(\langle I,I_{\scriptscriptstyle 2}\rangle)=H(\langle I,I_{\scriptscriptstyle 1}\rangle \cap \langle I,I_{\scriptscriptstyle 2}\rangle)=H(\langle I, I_{\scriptscriptstyle 1}*I_{\scriptscriptstyle 2}\rangle)=H(I),$$
where the first two equalities follow from Proposition \ref{ztp}(\ref{uis}).
Since $H(\langle I,I_{\scriptscriptstyle 1}\rangle)\neq H(I)$ and $H(\langle I,I_{\scriptscriptstyle 2}\rangle)\neq H(I),$ the closed set $H(I)$ is not irreducible, a contradiction.
\end{proof}

\begin{theorem}
Every irreducible closed subset of $\mathrm{Spec}\,A$ has a unique generic point.
\end{theorem}

\begin{proof}
The existence of generic point follows from Lemma \ref{T0lem} and Lemma \ref{lemprime}, whereas uniqueness part follows from Proposition \ref{t0a}. 
\end{proof} 

Note that if we have a decreasing chain of prime ideals, then the intersection of the ideals in the chain gives the minimal element, which is either a prime ideal or the zero ideal, and in the second case, with a vacuous argument, it is also a prime ideal. This confirms the existence of the minimal prime ideals of $A$ and hence allows us to obtain the following

\begin{lemma}\label{thmirre}
The irreducible components of $\mathrm{Spec}\,A$ are the closed sets $H(P)$, where $P$ is a minimal prime ideal of $A$.
\end{lemma}

\begin{proof}
By Proposition \ref{irrpp}(\ref{mircv}), irreducible components of $\mathrm{Spec}\,A$ are closed. If $P$ is a minimal prime ideal, then by Lemma \ref{lemprime}, $H(P)$ is irreducible. If $H(P)$ is not a maximal irreducible subset of $\mathrm{Spec}\,A$, then there exists a maximal irreducible subset $H(Q)$ with $Q\in  \mathrm{Spec}\,A$ such that $H(P)\subsetneq H(Q)$. This implies that $P\in H(Q)$ and hence $Q\subsetneq P$, contradicting the minimality property of $P$.
\end{proof}

\begin{proposition}
$\mathrm{Spec}\,A$ is irreducible if and only if the $\mathrm{Nil}\,A$ is prime.  
\end{proposition}

\begin{proof}
If $\mathrm{Nil}\,A$ is prime, then it is the minimum prime ideal. Therefore, $\mathrm{Spec}\,A$ has only one maximal irreducible component, namely $H(\mathrm{Nil}\,A),$  and since by Proposition \ref{irrpp}(\ref{mircv}) the maximal irreducible components cover $\mathrm{Spec}\,A,$ the irreducible component $H(\mathrm{Nil}\,A)$ must be equal to $\mathrm{Spec}\,A$. Hence $\mathrm{Spec}\,A$ is irreducible by Lemma \ref{thmirre}.
Conversely, suppose $\mathrm{Spec}\,A$ is irreducible. Then by Proposition \ref{irrpp}(\ref{mircv}), there exists  $P\in\mathrm{Spec}\,A$ such that  $H(P)=\mathrm{Spec}\,A.$ Hence $P\subseteq Q$ for all $Q\in \mathrm{Spec}\,A,$ and so is $\mathrm{Nil}\,A,$ \textit{i.e.},  $P\subseteq \mathrm{Nil}\,A$. By Definition \ref{radi}, $\mathrm{Nil}\,A\subseteq P$, and hence $\mathrm{Nil}\,A$ is prime. 
\end{proof} 

Before we discuss continuous maps, let us first recall some of the basic facts about skew brace homomorphisms and quotient skew braces.

\begin{definition}
\emph{A \emph{skew subbrace} $S$ of a skew brace $A$ is a subgroup of $(A,+)$ and $(A,\circ).$ If $A$ and $B$ are skew braces, a \emph{skew brace homomorphism} is a map $f\colon A\to B$ of $A$ into $B$ such that  $f(a+a')=f(a)+f(a'),$ $f(a\circ  a')=f(a)\circ f(a')$ for all $a,$ $a'\in A$. The \emph{kernel} of $f$ is  the subset $\mathrm{ker}\,f=\{ a\in A\mid f(a)=e\}$ of $A$. The \emph{image} of a skew brace homomorphism $f$ is the subset $\mathrm{im}\,f=\{b\in B\mid b=f(a)\;\text{for some}\; a\in A\}$ of $B$. Let $f\colon A\to B$ be a skew brace homomorphism and let $I,$ $J$ be ideals of $B$ and $A$ respectively. The \emph{contraction} of $I$ and the \emph{extension} of $J$ are respectively denoted by $I^c$ and $J^e$.}
\end{definition} 

\begin{lemma}\label{qst}
Let $I$ and $J$ respectively be ideals of skew left braces $A$ and $B.$ Let $f\colon A\to B$ be a skew left brace homomorphism. Then
\begin{enumerate}[\upshape (i)]

\item\label{avc} $\mathrm{ker}\,f,$ $f\inv(J)$ are ideals of $A$ and $\mathrm{im}\,f$ is a skew subbrace of $B;$

\item\label{qtsb}
$a\circ I=a+I$ for every $a\in A$ and  $A/I=\{a+I\mid a\in A\}$ is a skew left brace under the operation: $(a+I)+(b+I)=(a+b)+I\;\; \text{and}\;\; (a+I)\circ(b+I)=(a\circ b)+I;$

\item 
there is a bijection between ideals of $A/I$ and ideals of $A$ containing $I;$ 
\item\label{fisj}  $f(I*J)=f(I)*f(J)$ for all $I, J\in \mathrm{Idl}\,A.$ 
\end{enumerate}
\end{lemma}
 
\begin{proposition}\label{conmap}
Suppose $f\colon A\to B$ is a skew brace homomorphism and define the map $f^!\colon  \mathrm{Spec}\,B\to \mathrm{Spec}\,A$ by  $f^!(P)=f\inv(P)$, where $P\in\mathrm{Spec}\,B.$ Then
\begin{enumerate}[\upshape (i)]
\item every prime ideal of $A$ is a contracted ideal if and only if $f^!$ is surjective;

\item if every prime ideal of $B$ is an extended ideal then $f^!$ is injective; 

\item\label{sbcm} \label{contxr}  $f^!$ is    continuous;

\item\label{shme}  if $f$ is  surjective, then $\mathrm{im}(f^!)$ is homeomorphic to $H(\mathrm{ker}\,f)$;
 
\item\label{idkn}  the image of $f^!$ is dense in $\mathrm{Spec}\,A$ if and only if $\mathrm{ker}\,f\subseteq  \mathrm{Nil}\,A.$   
\end{enumerate}
\end{proposition}
  
\begin{proof}
(i). Let $P\in \mathrm{Spec}\,A$ and $P = Q^c$ for some $Q\in \mathrm{Spec}\,B.$ Then, $f^!(Q) = f\inv(Q) = P.$ Hence  $f^!$ is surjective.
Conversely, if $f^!$ is surjective, then for any $P \in \mathrm{Spec}\,A,$ we have $P= f^!(Q) = Q^c$, as required.

(ii). Let $Q\in \mathrm{Spec}\,B$ and $Q = P^e.$ Then, $f^!(Q) = P^{ec}\supseteq P.$ Suppose $f^!(Q) = f^!(Q').$ Then,
$P^{ec} = (P')^{ec}$, where $Q' = (P')^e,$ this implies $P^{ece} = (P')^{ece},$ which implies $P^e = (P')^e$ and that  $Q = Q'.$

(iii). To show $f^!$ is continuous, we first show that $f\inv(P)\in \mathrm{Spec}\,A,$ whenever $P\in \mathrm{Spec}\,B$.  Let $I*J\subseteq f\inv(P),$ where $I, J \in \mathrm{Idl}\,A.$ Then $f(I*J)\subseteq P,$ and by Lemma \ref{qst}(\ref{fisj}) we have $f(I)*f(J)\subseteq P$. Since $P$ is prime, by Definition \ref{pri} either $f(I)\subseteq P$ or $f(J)\subseteq P,$ and that implies either  $I\subseteq f\inv(P)$ or $J\subseteq f\inv(P)$. If $H(I)$ is a closed subset of $\mathrm{Spec}\,A$, then for any $Q\in \mathrm{Spec}\,B,$ we have:
\begin{align*}
Q\!\in\! (f^!)\inv (H(I))\!\Leftrightarrow\! f^!(Q)\in H(I)\!\Leftrightarrow\! f\inv(Q)\!\in\! H(I)\!\Leftrightarrow\! I\subseteq f\inv(Q)\!\Leftrightarrow\! Q\!\in\! H(\langle f(I)\rangle).
\end{align*}

(iv). By Lemma \ref{qst}(\ref{avc}), $\mathrm{im}(f^!)$ and $\mathrm{ker}\,f$ are ideals of $A$. We first show that $\mathrm{cl}(f^!(H(I)))=H(f\inv(I))$ for any $I\in \mathrm{Idl}\,B$ and for that it is sufficient to show: $f^!(H(I))=H(f\inv(I))$. We observe:
$$ 
P\in H(f\inv(I))\Leftrightarrow f(f\inv(I)) \subseteq f(P)
\Leftrightarrow  P = f\inv(f(P))=f^!(f(P))\in f^!(H(I)),$$
and by taking $I=0$, we obtain  $\mathrm{im}(f^!)=H(\mathrm{ker}\,f).$
This implies $f^!$ induces a continuous bijection between $\mathrm{im}(f^!)$ and $H(\mathrm{ker}\,f).$ The continuity of  
$(f^!)\inv$ also now follows from the above. Hence we have the desired homeomorphism.
 
(v).  Note that $\mathrm{im}(f^!)$ is dense in $\mathrm{Spec}\,A$ if and only if $\mathrm{cl}(\mathrm{im}\,f^!)=H(\mathrm{ker}\,f)=\mathrm{Spec}\,A$, and this occurs if and only if $\mathrm{ker}\,f\subseteq P$ for all $P\in \mathrm{Spec}\,A,$ and that holds if and only if $\mathrm{ker}\,f\subseteq \mathrm{Nil}\,A$.
\end{proof}

\begin{corollary}\label{homeo}
For a skew brace $A$, the spaces $\mathrm{Spec}\,A$ and $\mathrm{Spec}(A/\mathrm{Nil}\,A)$ are homeomorphic.
\end{corollary}

\begin{proof}
Note that by Lemma \ref{qst}(\ref{qtsb}), $A/\mathrm{Nil}\,A$ is a skew brace. On applying Proposition \ref{conmap}(\ref{shme}) on the homomorphism $A\to A/\mathrm{Nil}\,A$ we obtain the desired result. 
\end{proof}

\begin{proposition} 
Let $I$ be an ideal of a skew brace $A$ and $J=I^c$ be an ideal of a skew brace $A'.$ Let $\bar{\phi}\colon A/I\to A'/J$ be the skew brace homomorphism induced by the skew brace homomorphism $f\colon A\to A'.$ Then the restriction of the map $\phi_*\colon \mathrm{Spec}\,A'\to \mathrm{Spec}\,A$ to $H(J)$ is the map $\bar{\phi}_*\colon \mathrm{Spec}(A'/J)\to \mathrm{Spec}(A/I).$
\end{proposition} 

\begin{proof} 
Note that below the left commutative diagram induces the right commutative diagram.
\begin{center} 
$
\xymatrix{
A\ar[r]^{\phi}\ar[d]_{\pi^{I}} & A'\ar[d]^{\pi^{J}}\\
A/I\ar[r]^{\bar{\phi}} & A'/J
}
$\qquad
$
\xymatrix{
\mathrm{Spec}\,A &\mathrm{Spec}\,A'\ar[l]_{\phi_*}\\
\mathrm{Spec}(A/I)\ar[u]^{\pi^{I}_*}\ar[r]^{\bar{\phi_*}} &\mathrm{Spec}(A'/J).\ar[u]_{\pi^{J}_*}
}
$
\end{center}
Moreover, $K\in H(J)$ implies $\phi_*(K)=K^c\supseteq J^c=I^{ec}\supseteq I,$ that is, $\phi_*(H(J))\subseteq H(I).$
Now the desired result follows from Proposition \ref{conmap}(\ref{shme}). 
\end{proof}

\begin{definition} \emph{According to Jespers, Kubat, Van Antwerpen, \&  Vendramin \cite[Definition 4.1]{JKAV21}  
the weight $\omega(A)$ of non-zero skew brace $A$ is defined as
the minimal number of elements of A needed to generate $A$ (as an ideal). By convention,
we put $\omega(A)=1$ if $A = 0.$ }
\end{definition}

Recall that a skew left
brace A is said to be \emph{Noetherian} if every ascending chain of ideals of A is eventually
stationary. It follows immediately that a skew left brace is Noetherian if and only if all its ideals have
finite weight. Also recall that a topological space $X$ is said to be \emph{Noetherian} if the closed subsets of $X$ satisfy the descending chain condition. 

\begin{proposition}\label{fwn}
If all ideals of $A$ have
finite weight, then $\mathrm{Spec}\,A$ is a Noetherian space.  
\end{proposition}

\begin{proof}
It suffices to show that a collection of closed sets in $\mathrm{Spec}\,A$ satisfy descending chain condition. Let $H(I_1)\supseteq H(I_2)\supseteq \cdots$
be a descending chain of closed sets in $\mathrm{Spec}\,A$. Then by Lemma \ref{hihr}, $I_1\subseteq I_2\subseteq \cdots$ is an ascending chain of ideals in $A.$ Since $A$ is Noetherian, it stabilizes at some $n \in \mathds{N}.$ Hence, $H (I_n) = H(I_{n+k})$ for any $k.$ Thus $\mathrm{Spec}\,A$ is Noetherian.
\end{proof}

\begin{corollary}
The set of minimal prime ideals in a Noetherian skew brace is finite.
\end{corollary}

\begin{proof}
By Proposition \ref{fwn}, $\mathrm{Spec}\,A$ is Noetherian, thus $\mathrm{Spec}\,A$ has  finitely many irreducible
components. By Lemma \ref{thmirre}, every irreducible closed subset of $\mathrm{Spec}\,A$
is of the form $H(P),$ where $P$ is a minimal prime ideal. Thus $H(P)$ is irreducible components if and only if $P$ is minimal prime. Hence, $A$ has
only finitely many minimal prime ideals. 
\end{proof}

The following well-known result holds for any topological space and hence for $\mathrm{Spec}\,A$.
 
\begin{proposition}
If $A$ is a skew brace, then the following are equivalent:
\begin{enumerate}[\upshape (i)]
\item $\mathrm{Spec}\,A$ is Noetherian.
\item Every open subspace of $\mathrm{Spec}\,A$ is compact.
\item Every subspace of $\mathrm{Spec}\,A$ is compact.
\end{enumerate}
\end{proposition}

Recall that in the sense of \textsc{Hochster} \cite{H69}, a topological space is called \emph{spectral} if it is quasi-compact, sober,  admits a basis of quasi-compact open subspaces that is closed under finite intersections.

\begin{theorem}\label{spida}
$\mathrm{Spec}(\mathrm{Idl}\,A)$ is a spectral space.
\end{theorem}

\begin{proof}
Note that by Proposition \ref{mullat}, $\mathrm{Idl}\,A$ is a multiplicative lattice. Since every skew brace $A$ is finitely generated (by $e$) as an ideal of itself, by \textsc{Facchini, Finocchiaro, \& Janelidze} \cite[Theorem 11.5]{FFJ22} we have the desired claim.
\end{proof}

\begin{conrem}
	
\emph{The following are some natural topological questions that arise when we study spectral topology on a prime spectrum of an algebraic structure and we do not have answers to these questions for skew braces.}

\emph{$\bullet$ Although the partition of unity property is a sufficient condition for compactness of prime spectrum of a ring with identity (see \textsc{Dube \& Goswami} \cite{DG22}), we can not apply the same argument for skew braces. In the context of a skew brace $A$, when we try to apply the finite intersection property, although we obtain a finite sum representation of the identity $e$ of $A$, but from that we cannot conclude the same for an arbitrary element of $A$.}

\emph{$\bullet$ If $a$ is an element of a skew brace $A$ such that $a^2=a$, then we immediately see that $a=e$. This means a skew brace does not have nontrivial idempotent elements. For a commutative ring $R$ (with identity), without having any nontrivial idempotetnt elements implies $\mathrm{Spec}\,R$ is connected. Therefore, one may expect the same for a skew brace. Once again not having zero divisors and ``multiplicative identity'' in a skew brace lead to the problem of checking the connectivity.}

\emph{$\bullet$  \textsc{Hochster} \cite{H69} proved that a prime spectra of a commutative ring (with 1) endowed with spectral topology is spectral. We proved $\mathrm{Spec}(\mathrm{Idl}\,A)$ is spectral (see Theorem \ref{spida}), but it could be interesting to know whether $\mathrm{Spec}\,A$ itself is spectral.}

\emph{$\bullet$ Proposition \ref{fwn} gives a sufficiant condition for $\mathrm{Spec}\,A$ to be Noetherian. But it would be nice to characterize Notherian spaces as well as obtain a similar result for Artinian spaces.}
\end{conrem}

\section*{Acknowledgement}

The second author wishes to express his gratitude to  Leandro Vendramin
and Paola Stefanelli for some fruitful discussions.

\end{document}